\documentclass[11pt,reqno,tbtags]{amsart}
\usepackage{amssymb}
\numberwithin{equation}{section}

\newtheorem*{property*}{Property \csname @currentlabel\endcsname}
\newtheorem{theorem}{Theorem}[section]

\newtheorem{corollary}[theorem]{Corollary}

\theoremstyle{definition}
\newtheorem{example}[theorem]{Example}

\newtheorem{remark}[theorem]{Remark}

\theoremstyle{remark}

\newenvironment{romenumerate}{\begin{enumerate}% gives (i), (ii) etc.
 }{\end{enumerate}}

\newenvironment{abenumerate}{\begin{enumerate}% gives (i), (ii) etc.
 }{\end{enumerate}}

\newcounter{oldenumi}
{\setcounter{oldenumi}{\value{enumi}}
\begin{romenumerate} \setcounter{enumi}{\value{oldenumi}}}
{\end{romenumerate}}

\newcounter{thmenumerate}

\newcounter{xenumerate}   %no left indentation; thus wider lines

\newcommand\pfitemx[1]{\par\emph{#1}:}

\newcommand{\refT}[1]{Theorem~\ref{#1}}

\begingroup
  \divide\count255 by 60
  \multiply\count255 by -60
  \advance\count255 by \time
  \ifnum \count255 < 10 \xdef\klockan{\the\count1.0\the\count255}
  \else\xdef\klockan{\the\count1.\the\count255}\fi
\endgroup

\newcommand\xx{\mathbf{x}} % or something else to avoid clash with $x_i$
\newcommand\set[1]{\ensuremath{\{#1\}}}

\newcommand\xpar[1]{(#1)}
\newcommand\bigpar[1]{\bigl(#1\bigr)}
\newcommand\Bigpar[1]{\Bigl(#1\Bigr)}
\newcommand\biggpar[1]{\biggl(#1\biggr)}
\newcommand\lrpar[1]{\left(#1\right)}

\newcommand\xcpar[1]{\{#1\}}
\newcommand\bigcpar[1]{\bigl\{#1\bigr\}}

\newcommand\lrcpar[1]{\left\{#1\right\}}
\newcommand\bigabs[1]{\bigl|#1\bigr|}

\def\rompar(#1){\textup(#1\textup)}    % usage: \rompar(...)

\newcommand\parfrac[2]{\Bigpar{\frac{#1}{#2}}}

\def\xexp(#1){e^{#1}}
\newcommand\ceil[1]{\lceil#1\rceil}
\newcommand\floor[1]{\lfloor#1\rfloor}

\newcommand\ie{i.e.\spacefactor=1000}

\newcommand\viz{viz.\spacefactor=1000}
\newcommand\cf{cf.\spacefactor=1000}

\newcounter{CC}
\newcommand{\CC}{\stepcounter{CC}\CCx} %new constant C_i
\newcommand{\CCx}{C_{\arabic{CC}}}     %repeats the last C_i
\newcommand{\CCdef}[1]{\xdef#1{\CCx}}     %defines #1 as the last C_i
\newcounter{cc}
\newcommand{\cc}{\stepcounter{cc}\ccx} %new constant c_i
\newcommand{\ccx}{c_{\arabic{cc}}}     %repeats the last c_i
\newcommand{\ccdef}[1]{\xdef#1{\ccx}}     %defines #1 as the last c_i

\newcommand\E{\operatorname{\mathbb E{}}}
\renewcommand\P{\operatorname{\mathbb P{}}}

\newcommand\Bin{\operatorname{Bin}}

\newcommand\vol{\operatorname{vol}}

\newcommand\ga{\alpha}

\newcommand\gd{\delta}
\newcommand\gD{\Delta}

\newcommand\cE{\mathcal E}
\newcommand\cH{\mathcal H}
\newcommand\cS{{\mathcal S}}

\def\[#1]{[\![#1]\!]}

\newcommand\qq{^{1/2}}

\newcommand\qw{^{-1}}
\newcommand\qww{^{-2}}

\renewcommand{\=}{:=}

\newcommand\rhs{right hand side}

\newcommand{\gnp}{G(n,p)}
\newcommand{\erg}{e_R(G)}
\newcommand{\ex}[1]{e(#1)}

\newcommand{\vx}[1]{v(#1)}
\newcommand{\gax}{\ga^*}
\newcommand{\MRG}{M_{R,G}}
\newcommand{\tit}{\tilde t}
\newcommand{\ced}{\cE_\gd}
\newcommand{\nd}{N_\gd}
\newcommand{\xc}{a}
\newcommand{\xco}{a_0}
\newcommand{\ct}{c_1}

\newcommand\REM[1]{{\raggedright\texttt{[#1]}\par\marginal{XXX}}}

\newcommand\urladdrx[1]{{\urladdr{\def~{{\tiny$\sim$}}#1}}}

\begin{document}
\title[Upper tails for subhypergraphs and rooted random graphs]
{Upper tails for counting objects in randomly induced subhypergraphs and rooted random graphs}

\date{May 7, 2009} % (typeset \today{} \klockan)} 

\author{Svante Janson}
\address{Department of Mathematics, Uppsala University, PO Box 480,
SE-751~06 Uppsala, Sweden}
\email{svante.janson@math.uu.se}
\urladdrx{http://www.math.uu.se/~svante/}

\author{Andrzej Ruci\'nski}
\address {Department of Discrete Mathematics, Adam Mickiewicz
  University, Pozna\'n, Poland}
\email{rucinski@amu.edu.pl} 
\thanks{Second author supported by Polish grant N201036 32/2546. 
Research was performed while 
the authors visited Institut Mittag-Leffler in  Djursholm,
Sweden, during the program 'Discrete Probability', 2009.}

%\keywords{<keywords>}
%{60C05 (68P10,68W40)} %%{Primary: <subject>; Secondary: <subject>}

\begin{abstract}
General upper tail estimates are given for counting edges in a random
induced subhypergraph of a fixed hypergraph $\cH$, with an easy proof
by estimating the moments. As an application we consider the numbers of
arithmetic progressions and Schur triples in random subsets of
integers. In
the second part of the paper we return to the subgraph counts in
random graphs and provide upper tail estimates in the rooted case.
\end{abstract}

\maketitle

%\XXX

\section{Introduction}\label{S:intro}

Consider a finite sum of dependent random variables of the following form. Let $\Gamma$ be a finite
ground set and let $\cS$ be a family of its subsets. Let $\Gamma_p$ be a random, binomial subset of
$\Gamma$ which independently includes each element of $\Gamma$ with probability $p$. Finally, for
each $S\in\cS$, let $I_S$ be the indicator random variable of the event $\{S\subseteq\Gamma_p\}$.
Then $X=X(\Gamma,\cS,p)=\sum_{S\in\cS}I_S$ counts the number of members of the family $\cS$
contained in a random subset  $\Gamma_p$. A lot of research has been devoted to the study of the
asymptotic distribution of $X$ when the order $N=|\Gamma|$ grows to $\infty$ and $p=p(N)$, both in
a general setting and for particular instances, most notably for random graphs, see \cite{JLR}.

One feature which received a lot of attention is the rate of decay of
the tails of $X$, the lower tail $\P(X\le t\E X)$ for $0<t<1$, and the
upper tail $\P(X\ge t\E X)$ for $t>1$. Good estimates for the lower tail
follow from the FKG inequality (lower bound) and Janson's inequality
(upper bound), see \cite{JLR}, Section 2.2. Often, these two bounds asymptotically match under some
restrictions on the dependencies among the summands $I_S$. This is, in
particular, the case of subgraph counts in random graphs, see
\cite{JLR}, Section 3.1.

The upper tails tend to be harder to analyze. Some ad hoc results can
be found in \cite{JLR}, \cite{vu}, \cite{inf}, \cite{del}, among
others. For the subgraph count problem a quite satisfactory and
complete result has been obtained in \cite{JOR}, where the logarithms
of the upper and lower bound on $\P(X\ge t\E X)$ are of the same order of
magnitude except for a logarithmic term. A generalization to random
hypergraphs can be found in \cite{dujo}.

This paper can be viewed as a follow-up paper to \cite{JOR}. Using the
proof techniques developed therein, those results are extended in two
directions. First, we return to the more general model of set systems
(or hypergraphs) and obtain some straightforward estimates for the
upper tail of $X$, covering, in particular, the number of arithmetic
progressions of given length in a random subset of integers. Then, we
return to the subgraph counts to study the rooted version of the
problem, only to discover some unexpected features there.

%\section{Preliminaries}\label{Sprel}

\section{Counting edges of randomly induced subhypergraphs}

Let $\cH$ be a $k$-uniform hypergraph on a vertex set $\Gamma$ with
$|\Gamma|=N$ and with $|\cH|=\xc N^q$ edges, where $\xc=\xc(N)>0$ and $0<q\le
k$. Consider a random, binomial subset $\Gamma_p$ of $\Gamma$, where
$0<p=p(N)<1$, and the random variable $X=|\cH[\Gamma_p]|$ counting the
edges of $\cH$ that are entirely present in $\Gamma_p$. Note that
$$\mu:=\E X=|\cH|p^k=\xc N^qp^k.$$
For $j=0,1,\dots,k$, let
$$\Delta_j=\max_{S\in \binom{\Gamma}j}|\{T\in\cH: T\supseteq S\}|,$$
i.e., the maximum number of edges that contain $j$ given vertices.

\begin{theorem}\label{T1} Let $q$ be an integer, $1\le q\le k$, and let
$\xco>0$ and  $t>1$ be real numbers.
There exists a constant $c=c(q,\xco,t)$ such that if   $\cH$ satisfies the
following four conditions: 
\begin{romenumerate}
\item\label{t1xc}
 $\xc(N)=|\cH|/N^q\ge \xco$,
\item \label{t1j}
for all $j\le q$ we have $\Delta_j=O(N^{q-j})$,
\item \label{t1jq}
for all $j>q$ we have $\Delta_j=O(1)$, 
\item \label{t1Gamma0}
there exists $C>0$ and  $\Gamma_0\subseteq\Gamma$ such that
  $|\Gamma_0|\le C\mu^{1/q}$ and $|\cH[\Gamma_0]|\ge t\mu$,
\end{romenumerate}
then, with $X=|\cH[\Gamma_p]|$,
\begin{equation*}
p^{C\mu^{1/q}}=\exp\bigcpar{-C\mu^{1/q}\log(1/p)}
\le \P(X\ge t\mu)\le\exp\{-c\mu^{1/q}\}.
\end{equation*}
\end{theorem}

Before giving the proof, we
make some comments.

  \begin{itemize}
\item

The two exponents are of the same order
of magnitude  except for the
logarithmic term $\log(1/p)$; this inaccuracy disappears obviously for $p$ constant.

\item
Note that $\P(X\ge t\mu)>0\iff t\mu\le|\cH|\iff tp^k\le1$, so the
theorem is interesting for $t\le p^{-k}$ only. (For
larger $t$, $\P(X\ge t\mu)=0$ so the lower bound
fails, while the upper bound is trivial; further,
\ref{t1Gamma0} fails.)

\item
Condition \ref{t1jq} is redundant, since it follows from 
\ref{t1j} with $j=q$, but we prefer to include it
explicitly for emphasis, and for comparison with \refT{T2} which allows non-integer values of $q$ (note that for non-integer $q$, \ref{t1jq} does not follow from \ref{t1j}).

\item
As we will see in the proof, the upper bound follows only from conditions \ref{t1xc}--\ref{t1jq}, while the lower bound is a consequence of condition \ref{t1Gamma0} alone.

\end{itemize}

\begin{proof}
Take $C$ and $\Gamma_0$ as in assumption \ref{t1Gamma0}. We have
$$ \P(X\ge t\mu)\ge \P(\Gamma_p\supseteq \Gamma_0)= p^{|\Gamma_0|},$$
which proves the lower bound. 

For the upper bound, we use the same
approach as in \cite{JOR}. By Markov's inequality, for every $m$ we
have
$$ \P(X\ge t\mu)\le \frac{\E X^m}{t^m\mu^m}.$$
It remains to show that for a
 sufficiently small $\ct=\ct(q,\xco,t)$ and
$m=\ceil{\ct\mu^{1/q}}$
we have, say,  
$\E X^m\le t^{m/2}\mu^m$.

Having chosen $m-1$ (not necessarily distinct) edges 
$E_1,\dots,E_{m-1}$ of $\cH$, let $N_j$
be the number of edges $E_m$ such that
$\bigabs{E_m\cap\bigcup_{i=1}^{m-1}E_i}=j$,
and let $N_{\ge j}=\sum_{k\ge j} N_k$. We
estimate these numbers as follows:
For $j=0$, 
\begin{equation}
  \label{m0}
N_0\le N_{\ge0}=|\cH|.
\end{equation}
For $1\le j\le q$, by \ref{t1j},
\begin{equation}
  \label{ma}
N_j\le N_{\ge j} = O(m^j\gD_j)=O(m^jN^{q-j}),
\end{equation}
since if
$\bigabs{E_m\cap\bigcup_{i=1}^{m-1}E_i}\ge j$,
then there exists a set $A\subseteq \bigcup_{i=1}^{m-1}E_i$ with $|A|=j$
and $E_m\supseteq A$, and there are $O(m^j)$ such sets
$A$, and at most $\gD_j$ edges $E_m$ for each $A$.
For $j>q$ we  obtain
\begin{equation}\label{mb}
  N_j\le N_{\ge q} = O(m^q)
\end{equation}
from  \eqref{ma} (with $j=q$).

Arguing as in \cite{JOR} we have from
\eqref{m0}--\eqref{mb}, by induction on $m$, 
\begin{equation*}
  \begin{split}
\E X^m&\le 
\mu\lrpar{|\cH|p^k+\sum_{j=1}^qO(m^jN^{q-j})p^{k-j}
+\sum_{j=q+1}^kO(m^q)p^{k-j}}^{m-1}
\\
&=\mu^m\left(1+O(\xc\qw)\sum_{j=1}^q
 \left(\frac{m}{Np}\right)^j+O(1)\frac{m^q}{\mu}\right)^{m-1}
  \end{split}
\end{equation*}
for every $m\ge1$. Now choose
$m=\ceil{\ct\mu^{1/q}}\ge1$, as
said above.
If $m\ge2$, then
$m/(Np)\le2\ct \mu^{1/q}/(Np)=2\ct\xc^{1/q} p^{k/q-1}\le2\ct\xc^{1/q}$,
and thus, using \ref{t1xc},
the term in parenthesis in the last line can be made arbitrarily close
to 1 for all $m\ge2$ by choosing $\ct>0$ small enough;
in particular, it can be
made less than $t^{1/2}$.
Hence, for the chosen $m$, $\E X^m\le
t^{m/2}\mu^m$ if $m\ge2$, and trivially if
$m=1$ too.
This completes the proof.
\end{proof}

 In the case of non-integer $q$, the upper bound gets further away
 from the lower bound. Indeed, we then have the following result.

\begin{theorem}\label{T2} Let  $q$, $\xco$ and $t$ be real
 numbers, with $0< q\le k$, $\xco>0$  and  $t>1$.
There exists a constant $c=c(q,\xco,t)$ such that under the same assumptions
\ref{t1xc}--\ref{t1Gamma0} as in Theorem \ref{T1},
$$\P(X\ge t\mu)
\le\exp\left\{-c\max\left(\mu^{1/q}p^{k(1/\lfloor q\rfloor-1/q)},
\mu^{1/\lceil q\rceil}\right)\right\}$$
and
$$ \P(X\ge t\mu)\ge p^{C\mu^{1/q}}=\exp\{-C\mu^{1/q}\log(1/p)\}.$$
\end{theorem}

\begin{proof}
The only difference in the proof is  when we bound $N_j$ to estimate
$\E X^m$. Namely, for $j\ge\lceil q\rceil$, we either 
use $N_j\le N_{\ge\floor q}=O(m^{\floor q}N^{q-\floor q})$,
or 
$N_j\le N_{\ge\ceil q}=O(m^{\ceil q})$.
We then choose
$$m=\ceil{\ct\max\bigpar{\mu^{1/q}p^{k(1/\lfloor q\rfloor-1/q)},
\mu^{1/\lceil q\rceil}}}$$ 
for a small constant $\ct$.
(We may assume $\mu\ge1$, since otherwise $m=1$ and, recalling that $t>1$,
the estimate $\E X\le t^{1/2}\mu$ is trivial.) 
\end{proof}

%\bigskip

\subsection{Integer solutions of linear homogeneous systems}

For an $l\times k$ integer matrix $A$, where $l< k$, assume that every
$l\times l$ submatrix $B$ of $A$ has full rank $r(B)=l=r(A)$. Consider
the system of homogeneous linear equations $Ax=0$, where
$x=(x_1,\dots,x_k)$ is a column vector and $0$ is a column vector of
dimension $l$. We assume also that there exists a distinct-valued positive
integer solution of $Ax=0$. These assumptions seem to be quite
restrictive, but, in fact, we cover at least one important case: the
arithmetic progressions of length $k$ which can be viewed as
distinct-valued solutions to  a system of $l=k-2$ equations.

Let $\Gamma=[N]:=\{1,2,\dots,N\}$ and $0<p=p(N)<1$. Then $\Gamma_p$ is
 a random subset of the first $N$  integers with density $p$. Define
 a $k$-uniform hypergraph $\cH_A=\cH_A(N)$ as the family of all solution sets
 $\{x_1,\dots,x_k\}$ of the system $Ax=0$ with $x_i$ distinct and in $[N]$.
 Let us check that for some $a_0, q$, and $C$ the assumptions \ref{t1xc}--\ref{t1Gamma0} of Theorem
 \ref{T1} hold,  at least in the interesting case 
$\mu=|\cH_A|p^k\ge1$ and
$t\mu\le|\cH_A|$, which can be equivalently restated as
\begin{equation}\label{int}
\mu\ge1\quad\mbox{ and }\quad t\le p^{-k}.
\end{equation}
Set $q=k-l$.

\ref{t1xc}, \ref{t1Gamma0}:
We will show that  there exists $\xco>0$ such that for
sufficiently large $m\le N$  we have
\begin{equation}
  \label{magnus}
%|\cH_A\cap\tbinom{[m]}k|=
|\cH_A(m)|\ge \xco m^q. 
\end{equation}
Taking $m=N$ in \eqref{magnus} we obtain $|\cH_A|\ge\xco N^q$,
which is \ref{t1xc}.
Taking $m=\min\bigpar{\ceil{(t\xco\qw\mu)^{1/q}},N}$ in
\eqref{magnus} and $\Gamma_0=[m]$ we obtain
\ref{t1Gamma0} with $C=2(ta_0^{-1})^{1/q}$, 
using the assumptions in (\ref{int}).

Let $\xx_0\in Z^k$ be a positive integer solution of $Ax=0$. Let $M_0$ by the largest of its
coefficients  $x_{01},\dots,x_{0k}$. Let $\xx_1,\dots,\xx_q$ be $q$ linearly independent integer
solutions of $Ax=0$. (There exist $q$ linearly independent rational solutions, and we may multiply
these by their common denominators and thus assume that they are integer solutions.) Let $M$ be the
maximum of the absolute values of the coefficients in $\xx_1,\dots,\xx_q$.

Given $m$, let $d\=\floor{m/(M_0+1)}$. For any
integers $a_1,\dots,a_q$, the 
sum $d\xx_0+\sum_{i=1}^q a_i\xx_i$ yields an integer solution of $Ax=0$, and these solutions are
all distinct. If further $|a_i|<d/(2qM)$ for all $i$, this solution has all coefficients positive,
less than $m$, and distinct. The number of these solutions is
$\Theta(d^q)=\Theta(m^q)$. Hence, \eqref{magnus} holds.

\ref{t1j}, \ref{t1jq}: By elementary algebraic properties of systems of linear equations, every system $By=c$, where
$B$ is an integer $l\times h$ matrix, has no more than $N^{h-r(B)}$ solutions in $[N]$. Thus,
$\Delta_0=|\cH_A|\le N^{k-l}=N^{q}$. For every subset $J$ of the columns of 
$A$, define $A_J$ as the submatrix obtained from $A$ by removing the columns in $J$. This means
that when we fix values of some $j$ variables, then the obtained system of equations is of the form
$By=c$, where $y$ consists of the remaining unknowns, $B=A_J$, and $J$ is the set of columns of $A$
corresponding to the fixed variables. Hence, 
the number of solutions with $j$ given elements corresponding to
the given columns $J$ is at most 
$N^{k-j-r(A_J)}$. Now, for all $j\le
q=k-l$, if $|J|=j$ then, by our assumption on $A$,  $r(A_J)=l$, so
(summing over $J$) $\Delta_j=O(N^{k-j-l})=O(N^{q-l})$. On the
other hand, if $j>k-l$ then $r(A_J)=k-j$, so 
$\Delta_j=O(N^0)=O(1)$.

Hence, given (\ref{int}), Theorem \ref{T1} applies for such $\cH_A$ with $q=k-l$ and
$\mu=\Theta(N^{k-l}p^k)$.

\begin{example}\label{AP} In particular, we
 obtain quite sharp estimates for the tails of the numbers of arithmetic
 progressions of length $k$ in  $[N]_p$. Indeed, they are given by the system
$x_i-2x_{i+1}+x_{i+2}=0$, $i=1,\dots,k-2$. It is easy to check that
 for $l=k-2$ every $l\times l$ 
submatrix has full rank, and we have the following result.
\begin{corollary}\label{APc} Let $X$ be the number of arithmetic progressions
 of length $k$ in $[N]_p$, $k\ge 3$, 
and let $\mu$, $t>1$, and $p$ satisfy \eqref{int}. 
Then there exist
 $c,C>0$ such that 
 \begin{equation*}
p^{CNp^{k/2}}=\exp\{-CNp^{k/2}\log(1/p)\}\le \P(X\ge
t\mu)\le\exp\{-cNp^{k/2}\}.
\eqno{\qed}
 \end{equation*}
\end{corollary}
\end{example}

\begin{example} A Schur
 triple is a triple $\{x,y,z\}$ of positive integers such that $x+y=z$, $x\neq
 y$. In this case we have $k=3$, $l=1$ and so, $q=2$.
\begin{corollary}\label{Sch} Let $X$ be the number of Schur triples in
 $[N]_p$, 
and let $\mu$, $t>1$, and $p$ satisfy \eqref{int} with $k=3$. 
Then there exist $c,C>0$ such that 
$$ p^{CNp^{3/2}}=\exp\{-CNp^{3/2}\log(1/p)\}\le \P(X\ge
t\mu)\le\exp\{-cNp^{3/2}\}.
\eqno{\qed}
$$ 
\end{corollary}
\end{example}

\begin{remark}
Arithmetic progressions are partition regular, a name introduced by Rado for all linear systems the
solutions of which satisfy theorems  similar to the van der Waerden theorem. But, in addition, they
are also density regular, which means that every subset of integers of positive density contains
them (Szemer\'edi's theorem). Partition properties of random
 subsets of integers with respect to density regular systems were studied  in \cite{RR}.
 Schur triples form an example of partition regular but not density regular linear system. Partition properties of random
 subsets of integers with respect to Schur triples were studied  in \cite{GRR}.
\end{remark}

\begin{remark}
  We have here treated the set of solutions $x$ to $Ax=0$
  as a hypergraph, \ie, we have treated the solutions $x$ as
  $k$-sets rather than $k$-vectors. This is fine for the
  examples of arithmetic progressions and Schur triples treated above,
  but in general it may be more natural to regard the solutions
  $x$ as vectors (or, equivalently, sequences) in $[N]^k$,
  rather than as sets. We then define $\cH_A$ as the subset
  $\set{x:Ax=0}$ of $[N]^k$. In this way, we
  distinguish between solutions that are permutations of each other
  (for example, $(x,y,z)$ and $(y,x,z)$ in the Schur
  triple case), and we allow repeated values. 

It is possible to prove a version of \refT{T1} for this case,
using essentially the same proof, but the possibility of repeated
elements of $\Gamma=[N]$ complicates the conditions; we now
need  bounds on the number of vectors in $\cH_A$ that have
$j$ coordinates fixed, and at most $\ell$ distinct
values of the other coordinates. We omit the details.
\end{remark}

\subsection{Further examples and remarks}

\begin{example}
 In the dense case, that is, when $q=k$, assumption \ref{t1Gamma0} holds
 trivially by averaging over all subsets $\Gamma_0$ of a
 suitable size, provided the necessary condition $t\le
 p^{-k}$ is satisfied, but this result has been known already
 (\cf{} \cite{inf} and \cite{del}). In particular, this case covers the
 number of matchings of size $k$ in a random $r$-uniform hypergraph
 $G^{(r)}(n,p)$, by considering a $k$-uniform hypergraph $\cH$ where
 the vertices are the edges of the complete $r$-uniform hypergraph
 $K_n^{(r)}$ and the edges are the matchings of size $k$ in
 $K_n^{(r)}$. Then the assumptions of Theorem \ref{T1} hold with
 $q=k$.
\end{example}

\begin{remark}
It can be very hard to improve upon Theorem
 \ref{T2}, because it contains the triangle count problem from
 \cite{JOR}. Indeed,  with $\Gamma=\binom{[n]}2$ and $\cH$ being the
 family of the edge sets of all triangles in $K_n$, we have
$N=\binom n2$ and 
 $|\cH|=\Theta(n^3)=\Theta(N^{3/2})$, so $q=3/2$. To get the result
 from \cite{JOR},  we would need to improve the upper bound, but this
 seems to be impossible without ``seeing'' the vertices of the random
 graph.
\end{remark}

\section{Rooted subgraphs of random graphs}

A \emph{rooted graph} $(R,G)$ is a  graph $G$ with a
fixed independent set $R$; we also say that the graph is
% is dubbed to be
\emph{rooted at $R$}. (For simplicity, we sometimes use $G$ to
 denote the rooted graph $(R,G)$ when $R$ is clear from the
 context.)
Counting rooted subgraphs of a
random graph $G(n,p)$ with a fixed set $R$ of roots plays an important
role in studying the so called extension statements and 0--1 laws in
random graphs, see, e.g., \cite[Sections 3.4 and 10.2]{JLR}. Another
application can be found in \cite{Cops}, where a sharp concentration
of the number of paths of given length connecting two given vertices
is utilized.
Here we give a quite accurate estimate of the upper tail of the number
of rooted copies of a given rooted graph in $G(n,p)$; the result is similar
to our main result in \cite{JOR} for unrooted graphs, but somewhat
simpler, except for a new complication for constant $p$.

A rooted graph $(R',H)$ is a \emph{rooted subgraph} of
$(R,G)$ if $H$ is a subgraph of $G$ and
$R'=V(H)\cap R$. We let $N^R(G,H)$ denote the number of rooted
copies of $H$ in $G$.

Given a rooted graph $(R,G)$ and  a graph $F$ on the
vertex set $V(F)=[n]=\{1,2,\dots,n\}$,
let $r=|R|$ and regard $F$ as rooted on $[r]=\set{1,\dots,r}$;
we say that a rooted subgraph of $([r],F)$ isomorphic to
$(R,G)$ is an
\emph{$R$-rooted copy of $G$ in $F$}.
Thus $N^R(F,G)$ is the number of
$R$-rooted copies of $G$ in $F$.
In particular, when $F$ is a random graph $G(n,p)$, we
let the random variable
$X=X^R_G=X_G^R(n,p)$ be the number $N^R(G(n,p),G)$ of
{$R$-rooted copies of $G$ in $G(n,p)$}.
We further define
\begin{equation}
  \label{mu}
\mu=\mu_R(G,n,p)\=\E X_G^R=N^R(K_n,G)p^{\ex G}.
\end{equation}

For a subgraph $H$ of $G$ let $H-R$ be the graph obtained from $H$ by
deleting all vertices of $R$ (together with incident edges), and
define
\begin{equation}
  \label{Psi}
\Psi_H^R=\Psi_H^R(n,p)\=n^{v\xpar{H-R}}p^{\ex H}.
\end{equation}
Note that $\Psi_H^R=\Theta(\E X_H^{R'})$, with
$R'=R\cap V(H)$, but as defined, it does not depend on the actual set $R'$ of roots of $H$.

Recall that, for a graph $H$,  the fractional independence number
$\alpha^*(H)$ is defined as the maximum value of  $\sum_ix_i$ over all
assignments  $(x_i)_{i\in V(H)}$ such that $0\le x_i\le1$ for all
vertices $i\in V(H)$ and  $x_i+x_j\le 1$ for every edge $ij\in H$.
We let
\begin{equation}
  \label{MRG}
\MRG=\MRG(n,p)
=\min_{H\subseteq G, \ex H>0}\left(\Psi_H^R\right)^{1/\alpha^*(H-R)}.
\end{equation}
%and $$\tilde M_{R.G}=\max(\MRG,1),$$
We further let
\begin{equation}
    \label{mRG}
m_R(G)\=\max_{H\subseteq G, \ex H>0}\frac{\ex H}{v(H-R)}>0,
\end{equation}
and note that \eqref{MRG}, \eqref{Psi} and
\eqref{mRG} imply that
\begin{equation}\label{m1}
  \MRG < 1 \iff n p^{m_R(G)} <1.
\end{equation}
 By the same argument as for the unrooted case
in \cite[Section 3.1]{JLR}, it is easy to show that $p=n^{-1/m_R(G)}$
is the threshold for the appearance of an $R$-rooted
copy of $G$ in $G(n,p)$.

Let $\erg=\ex G-\ex{G-R}$ be the number of edges in $G$
incident with the root set $R$.
We assume below that $\erg>0$; the case $\erg=0$ is
uninteresting since then $X^R_G$ equals the number of copies of
the unrooted graph $G-R$ in
$\gnp-[r]$, which we identify with $G(n-r,p)$,
so $X^R_G(n,p)=X_{G-R}(n-r,p)$
and we may apply the results of \cite{JOR}.
%(Formally, $G(n-r,p)$ is defined with vertex set $[n-r]$
%and $\gnp-R$ with vertex set

\begin{theorem}\label{T3}
For every rooted graph $(R,G)$ with $\erg>0$ and for every
  $t>1$ there exist constants $c=c(t,G)$ and
  $C=C(t,G)$ such that for all $n\ge \vx G$,
with $p_1\=t^{-1/\erg}$ and $p_2\=t^{-1/\ex G}$:
%and thus  $0<p_1\le p_2<1$.
\begin{abenumerate}
  \item
If $p\le n^{-1/m_R(G)}$,
then
\begin{equation*}
p^{C}=\exp\{-C \log(1/p)\}\le \P(X_G^R\ge t\mu)\le\exp\{-c\}.
\end{equation*}

\item
If $n^{-1/m_R(G)}\le p\le p_1$,
then
\begin{equation*}
p^{C \MRG}=\exp\{-C \MRG\log(1/p)\}
\le \P(X_G^R\ge t\mu)
\le\exp\{-c \MRG\}.
\end{equation*}
\item
If $p_1\le p\le p_2$,
then
\begin{equation*}
\exp\{-C(n+(p-p_1)^ 2n^ 2)\}
\le \P(X_G^R\ge t\mu)
\le\exp\{-c(n+(p-p_1)^ 2n^ 2)\}.
\end{equation*}
\item
If $p_2< p\le 1$,
then
\begin{equation*}
\P(X_G^R\ge t\mu)=0.
\end{equation*}
\end{abenumerate}
\end{theorem}

Note that  $0<p_1\le p_2<1$, and that $p_1$ and
$p_2$ do not depend on $n$. Before giving the proof, we
make some comments.

  \begin{romenumerate}
\item\label{t3a}
  Case (d) is trivial, because $p>p_2\iff t
  p^{e(G)}>1\iff t\mu>N^R(K_n,G)$, see \eqref{mu},
so it is impossible
  to get at least $t\mu$ rooted copies of $G$ on $n$
  vertices.

\item\label{t3b}
Case (a) is uninteresting and included only to
  show that the estimates in (b) extend in a continuous way to smaller
  $p$.
(Note that $\MRG=1$ at the threshold $p=
  n^{-1/m_R(G)}$, \cf{} \eqref{m1}.)
Indeed, in case (a) we are below the threshold, so typically
  $X^R_G=0$.

\item\label{t3c}
If $e_R(G)=e(G)$, or equivalently $e(G-R)=0$, \ie,
   all edges in $G$ have a root as one endpoint, then
$p_1=p_2$ and case (c) disappears, so that (b) is valid
  until the cutoff at $p_2$. For all other $G$,
  $p_1<p_2$ and case (c) appears, so there is a phase transition at $p_1$.

\item\label{t3d}
In the unrooted case in \cite{JOR} there is also a phase
transition at $p=n^{-1/\Delta_G}$. This has no
counterpart in the rooted case.

\item\label{t3mn} 
Since $\erg>0$, $G$ has a rooted subgraph $H_0$ which is just a single 
edge with one endpoint in $R$;
we have $\Psi_{H_0}^R=np$ and  $\alpha^*(H_0-R)=\alpha^*(K_1)=1$, so
\begin{equation}\label{lessn}
\MRG\le \left(\Psi_{H_0}^R\right)^{1/\alpha^*(H_0-R)}=np\le n.
\end{equation}
Hence, the upper bound in (b) is never stronger than $\exp\xcpar{-\Theta(n)}$.
\item\label{t3log} 
In (b)  the exponents in the lower and upper bound are of the same order
of magnitude  except for the
logarithmic term $\log(1/p)$; this inaccuracy disappears obviously for $p$ constant.

\item\label{t3e}
For any fixed $p>0$ (or $p=p(n)\in[p_0,1]$ for some constant  $p_0>0$),
$\Psi^R_H=\Theta(n^{v(H-R)})$. Since
$\gax(H-R)\le v(H-R)$ for all $H\subseteq G$,
with equality for at least one $H$ with $e(H)>0$,
\viz{} a single rooted edge, \eqref{MRG} shows that
  then $\MRG=\Theta(n)$. Consequently, the result in
  (b) can  be written for constant $p\le p_1$ as
$\P(X_G^R\ge t\mu)=\exp\set{-\Theta(n)}$. This shows that the
  bounds in (b) and (c) agree at $p=p_1$.
Moreover, we obtain the following corollary.
  \end{romenumerate}
%\end{remark}

\begin{corollary}\label{C3}
With assumptions and notations
as in \refT{T3}, assume further that $p$ is  fixed.
  \begin{abenumerate}
\item
If $0< p\le p_1$,
then
\begin{equation*}
\P(X_G^R\ge t\mu)=\exp\set{-\Theta(n)}.
\end{equation*}
\item
If $p_1< p\le p_2$,
then
\begin{equation*}
\P(X_G^R\ge t\mu)=\exp\set{-\Theta(n^ 2)}.
\end{equation*}
\item
If $p_2< p\le 1$,
then
\begin{equation*}
\P(X_G^R\ge t\mu)=0.
\end{equation*}
\end{abenumerate}
\end{corollary}

The sudden jump in the exponent from $n$ to $n^2$ at $p=p_1$
(for $G$ with $e(G-R)>0$, so $p_1<p_2$) may be
surprising, and has no counterpart in the unrooted case in
\cite{JOR}.
It may roughly be explained as follows (see the proof):
If $p<p_1$, then it suffices (typically) to have all $\Theta(n)$ edges from
the roots present in $\gnp$ in order to have more than $t\mu$ rooted
copies of $G$. However, if $p>p_1$, this is not enough,
and we need also (typically) a larger proportion than $p$ of the $\binom{n-r}2$
other possible edges, which by the usual Chernoff bound has
probability only $\exp\set{-\Theta(n^2)}$.

\begin{proof}[Proof of \refT{T3}]
We mostly follow closely the proof for the unrooted case from
\cite{JOR}, and therefore omit some details. As remarked above, (d) is
trivial. Part (a) can be proved by a modification of the argument
below, replacing $\MRG$ by 1; we omit the details and refer
to the corresponding argument in \cite{JOR}. Hence we
consider only (b) and (c).
We let $C_1,C_2\dots$ and $c_1,c_2,\dots$ denote
constants that may depend on $G$ and $t$, but not on
$n$ or $p$.
%($C_i$ are large and $c_i$ are small.)

\pfitemx{Upper bounds} If $(R,H)$ is a rooted graph, let $N^R(n,m,H)$ be the maximum of $N^R(F,H)$
over all rooted graphs $F$ with $v(F)\le n$ and $e(F)\le m$ and with a set of roots of size $|R|$.
In other words, $N^R(n,m,H)$ is the maximum number of copies of $(R,H)$ that can be packed in $n$
vertices and $m$ edges with a given set of $|R|$ roots.

Let us start with the observation that if the minimum degree $\delta(H)>0$ then
\begin{equation}\label{NRN1}
N^R(n,m,H)\le N^R(2m,m,H)=O(N(2m,m,H-R)).
\end{equation}
Indeed, for any $F$ with  $\vx F\le n$, $\ex F\le m$, and $\delta(F)>0$, we have $\vx F\le 2m$, so the left hand side
inequality follows. To prove the right hand side inequality,  assume that $F$ and $H$ have the same set of
roots $R$. Then
$$N^R(F,H)\le N(F-R,H-R)\times 2^{|R|(\vx H-|R|)}=O(N(2m,m,H-R)).$$

Now, to prove the upper bound on $\P(X_G^R\ge t\mu)$, as before, we
want to show that, say,  $\E X^m\le t^{m/2}\mu^m$, where $X=X_G^R$,
$\mu=\E X$, and $m$ is suitably large.
Similarly as in \cite{JOR} and, as a matter of fact, similarly to the
proof of Theorem \ref{T1} here, an inductive argument yields, for all
$m\ge1$,
\begin{equation}\label{exm}
  \E X^m
\le \mu^m\biggpar{1+\CC \sum_{H\subseteq G}
  \frac{N^{R'}\bigpar{n,(m-1)e(G),H}}{\Psi^R_H}}^{m-1},
\end{equation}
where the sum extends over all rooted subgraphs $(R',H)$ of $(R,G)$ with $\gd(H)>0$.
($H$ corresponds to the subgraph spanned by the edges in the 
intersection of the $m$th copy of $G$ and the union of the
$m-1$ previous copies, and as such has $\delta(H)>0$.)

We take $m\= \ceil{\cc\MRG}\ccdef\cca$ for a suitable small
constant $\cca\in(0,1)$ to be fixed later. 
By \eqref{NRN1}, \cite[Theorem 1.3]{JOR} and \eqref{MRG}, for every
$H\subseteq G$ with $\gd(H)>0$, 
assuming $m\ge2$,
\begin{equation*}
  \begin{split}
N^{R'}\bigpar{n,(m-1)e(G),H} &\le \CC N\bigpar{2(m-1)e(G),(m-1)e(G),H-R}
\\&
=\Theta(m^{\alpha^*(H-R)})
%\\&
=\Theta\bigpar{(\cca\MRG)^{\alpha^*(H-R)}}
\\&
\le \CC \cca \Psi^R_H.
%\le \CC \cca \mu_{R'}(H).
  \end{split}
\end{equation*}
Hence, \eqref{exm} yields (the case $m=1$ being
trivial),
$\E X^ m \le \mu^ m (1+\CC \cca)^{m-1}$. We choose $\cca$ so small
that $1+\CCx \cca\le t\qq$, and
then
Markov's inequality yields
\begin{equation}\label{ml4}
  \P(X\ge t\mu)\le \frac{\E  X^m}{t^m\mu^m}
\le t^{-m/2} \le \exp\set{-\cc\MRG}.
\end{equation}
In particular, this yields the upper bound in (b).

For the  upper bound in (c), we note that each rooted copy of
$G$ in $K_n$ yields a copy of $G-R$ in
$K_n-R=K_{n-r}$; conversely each copy of $G-R$ in
$K_{n}-R$ can be extended to exactly $g$ rooted
copies of $G$ in $K_n$, for some integer $g\ge1$
depending on $G$.
Hence,
$X^R_G(n,p)\le g X_{G-R}(n-r,p)$.
Further, $N^R(K_n,G)=g N(K_{n-r},G-R)$ so
\begin{equation}\label{ml0}
  \begin{split}
\mu&
%=  \mu_R(G,n,p)
=N^R(K_n,G)p^{e(G)}
=g N(K_{n-r},G-R)p^{e(G-R)+\erg}
\\&
= g \mu(G-R,n-r,p)p^{\erg}.
  \end{split}
\end{equation}
Consequently,
\begin{equation}\label{ml1}
  \begin{split}
  \P(X^R_G\ge t\mu)
&\le
\P\bigpar{gX_{G-R}(n-r,p)\ge t g
  \mu(G-R,n-r,p)p^{\erg}}
\\&
=
\P\bigpar{X_{G-R}(n-r,p)\ge tp^{\erg}  \mu(G-R,n-r,p)}.
  \end{split}
\end{equation}
Let $\tit\=tp^{\erg}$, and note that, for (c),
$1\le\tit\le t$.
By \cite[Theorems 1.2 and 1.5, and Remark 8.2]{JOR}, recalling
that $t$ is fixed and $p\ge p_1$,
\begin{equation}\label{ml2}
  \P\bigpar{X_{G-R}(n-r,p)\ge \tit \mu(G-R,n-r,p)}
\le \exp\set{-\cc(\tit-1)^2 n^2}. % p^{\Delta_{G-R}}}.
\end{equation}
Further,
\begin{equation*}
  \tit-1=tp^{\erg}-1=(p/p_1)^{\erg}-1\ge p/p_1-1\ge p-p_1,
\end{equation*}
so \eqref{ml1}--\eqref{ml2} yield
\begin{equation}\label{ml3}
  \P(X^R_G\ge t\mu) \le \exp\set{-\ccx(p-p_1)^2n^2}.
\end{equation}
The upper bound in (c) now follows by taking the geometric mean of
\eqref{ml4} and \eqref{ml3}, noting that in this range of $p$,
$\MRG=\Theta(n)$ as remarked in \ref{t3e} above.

\pfitemx{Lower bounds}
Let $H$ be a subgraph of $G$ such that $\ex H>0$ and
$$M:=\MRG= \left(\Psi_{H}^R\right)^{1/\alpha^*(H-R)}.$$
Since we consider parts (b) and (c) only,  $M\ge1$ by \eqref{m1}.

Set $p_0=(3v_Gt)^{-1}$ and assume first that $p\le p_0$. (Note that
$p_0<t\qw\le p_1$.)
We construct, as in \cite{JOR}, a graph $F$ with
\begin{equation}\label{con}
\vx F\le 3(v_G-r)tM,\quad \ex F=O(M),\quad\mbox{ and }\quad N(F,H-R)\ge 2t\Psi^R_H.
\end{equation}
This is done as follows. Let $(x_i)_{i\in V(H-R)}$ be an optimal
assignment for the fractional independence problem, that is, $0\le
x_i\le1$, $x_i+x_j\le 1$ for every edge $ij\in H-R$, and
$\sum_ix_i=\alpha^*(H-R)$.
Construct $F$ by blowing up each vertex of $H-R$ to a set of $\lceil
2tM^{x_i}\rceil$ vertices and replacing each edge of $H-R$ by the
complete bipartite graph. This yields (\ref{con}), where we have put 3
rather than 2 because of the ceiling. Now, by
\eqref{lessn},
$$\vx F\le 3(v_G-r)tM\le 3(v_G-r)tnp\le (1-r/v_G)n\le n-r.$$

We may thus
fix a copy $F_1$ of $F$ with $V(F_1)\subseteq[n]\setminus[r]$; we
further let
$F_2$ be $F_1$ enlarged by adding all roots $1,\dots,r$ together with
all $r\vx{F_1}=O(M)$ edges between the roots and $V(F_1)$. Now, exactly
as in \cite{JOR}, it follows from \cite[Lemma 3.3]{JOR} that
$$\P(X_G^R\ge t\mu)\ge\frac14p^{\ex G}\P(G(n,p)\supseteq
F_2)=\frac14p^{\ex G+\ex{F_2}}=p^{\Theta(M)}.$$
This proves the lower bound in (b) when $p\le p_0$.

Assume now that $p_0\le p\le p_2$ and note that  the lower
bound we want to prove can be written as
$\exp\set{-\Theta(n)}$, see \ref{t3e} above or Corollary \ref{C3}.

Consider first the case $e(G-R)=0$ and observe that then 
the maximum number
of copies of $G$ are obtained as soon as all edges from the
roots appear, so, denoting
this event by $\cE^R$,
\begin{equation*}%\label{allt}
  \P(X^R_G\ge t\mu) \ge \P\bigpar{\cE^R}
  = p^{r(n-r)}
\ge e^{-\CC n},
\end{equation*}
which proves the lower bound in (b) in this case.
(Since $e(G-R)=0$ implies $p_1=p_2$, (c) is trivial.)

Thus, it remains to consider the case when  $p_0\le p\le p_2$ and $e(G-R)>0$.
We note first the trivial bound
\begin{equation}\label{allt}
  \P(X^R_G\ge t\mu) \ge \P\bigpar{G(n,p)=K_n}
  = p^{\binom n2}
\ge e^{-\CC n^2}.
\end{equation}
 Let $Z$ be the
number of edges in $G(n-r,p)$. Since $Z$ has binomial
distribution $\Bin\bigpar{\binom{n-r}2,p}$ with mean $p\binom{n-r}2$,
it is easily seen that if $(1+3\gd)p\le1$,
and $\ced$ is the event \set{Z\ge(1+3\gd)p\binom{n-r}2},
then
\begin{equation}\label{ced}
\P(\ced)\ge \cc \exp\set{-\CC n^2 \gd^2}.   \ccdef\cced\CCdef\CCed
\end{equation}
(The Chernoff bounds are essentially sharp, as is easily seen using 
Stirling's formula.)
The number $X_{G-R}(n-r,p)$ of copies of $G-R$ in
$G(n-r,p)$ is a sum of $N(K_{n-r},G-R)$ indicator
variables $I_\ga$. Conditioned on $Z=z$, each of them has the
expectation
\begin{equation}\label{er1}
  \P(I_\ga=1\mid Z=z)=\frac{(z)_{e(G-R)}}{\lrpar{\binom{n-r}{2}}_{e(G-R)}}
=\biggpar{\frac{z}{\binom{n-r}{2}}}^{e(G-R)}\Bigpar{1-O(z\qw)}.
\end{equation}
Let $\nd\=(1+3\gd)p\binom{n-r}2$. If
$\gd\ge n\qw$,
$z\ge\nd$ and $n$ is large enough, then
\eqref{er1} yields
\begin{equation*}
  \begin{split}
  \P(I_\ga=1\mid Z=z)
&\ge
(1+3\gd)^{e(G-R)} p^{e(G-R)}\Bigpar{1-O(n^{-2})}
\\
&\ge
(1+2\gd) p^{e(G-R)}.
  \end{split}
\end{equation*}
Consequently, if $\gd\ge n\qw$ and $n$ is large
enough, then
$ \P(I_\ga=1\mid \ced)\ge(1+2\gd) p^{e(G-R)}$, and summing
over $\ga$ we find
%with $\mu(G-R)=\E X_{G-R}(n-r,p)$,
$$\E(X_{G-R}(n-r,p)\mid\ced)\ge(1+2\gd)\E X_{G-R}(n-r,p)
=(1+2\gd)\mu({G-R})
.$$
%It now follows from
Hence, by Lemma 3.2 of \cite{JOR}, as in the proof of Lemma  3.3
therein, with $1/2$ replaced by $\frac{1+\delta}{1+2\delta}$, we
obtain
$$
\P\bigpar{X_{G-R}\ge(1+\delta)\mu({G-R})\mid\ced}
\ge
\parfrac{\delta}{1+2\gd}^2\frac{\mu({G-R})}{N(K_{n-r},G-R)}
%=\delta^2p^{\ex{G-R}}
\ge \cc\delta^2. \ccdef\ccgml
$$
Assuming also the presence of all edges from the roots,
\ie, the event $\cE^R$, we have
$X^R_G=g X_{G-R}$ (where $g$ is as in the proof of the
upper bound); further, by \eqref{ml0},
$\mu= g \mu(G-R)p^{\erg}$; hence the inequality
$X_{G-R}\ge(1+\delta)\mu({G-R})$  is equivalent to
%$X^R_G\ge(1+\delta)p^{-\erg}\mu$.
\begin{equation}\label{er2}
X^R_G\ge(1+\delta)p^{-\erg}\mu.
\end{equation}
Consequently,
\begin{equation*}
  \begin{split}
\P\bigpar{X^R_G\ge(1+\delta)p^{-\erg}\mu\mid\cE^R,\ced}
\ge
\P\bigpar{X_{G-R}\ge(1+\delta)\mu({G-R})\mid\ced}
\ge\ccgml\gd^ 2
  \end{split}
\end{equation*}
and thus, by \eqref{ced},
\begin{equation*}
  \begin{split}
\P\bigpar{X^R_G\ge(1+\delta)p^{-\erg}\mu}
&\ge
 \ccgml\gd^2\P(\ced\text{ and }\cE^R)
=
 \ccgml\gd^2\P(\ced)\P(\cE^R)
\\&
\ge\cc n\qww \exp\{-\CCed\delta^2n^2\}p^{rn}=\exp\{-\Theta(\delta^2n^2+n)\},
  \end{split}
\end{equation*}
provided $1/n\le\delta\le\frac13(p^{-1}-1)$ and $n$ is large enough.

For $p_0\le p\le p_1$, we choose
$\gd=n\qw$; then the \rhs{} of
\eqref{er2} is greater than
$p_1^{-\erg}\mu=t\mu$,
so we obtain
$$\P(X^R_G\ge t\mu)\ge\exp\{-\Theta(n)\},$$
which as remarked above is equivalent to the lower bound in (b) for
this range of $p$.

Finally, if $p_1\le p\le p_2$, we take
$$\gd\=\max\left\{tp^{\erg}-1,1/n\right\}
 =\max\left\{(p/p_1)^{\erg}-1,1/n\right\}=\Theta(p-p_1+1/n),$$
so that the \rhs{} of \eqref{er2} is again at least $t\mu$.
This yields the lower bound in (c) when $n$ is large enough and $p\ge
 p_1$ is small enough to guarantee that $\delta\le\frac13(p^{-1}-1)$.
For larger $p$, as well as for small $n$, we simply  use \eqref{allt}.
This completes the proof of the lower bound in (c).
\end{proof}

\subsection{Examples and remarks}

It is easy to see that the minimum defining $M=\MRG$ in
\eqref{MRG} is achieved by a
subgraph $H$ of $G$ such that $H-R$ is connected and, for every vertex
$v\in H$, $H$ contains all edges leading from $v$ to $R$.  
These observations simplify 
computations of the bounds in Theorem \ref{T3}.

\begin{example}
 {\bf Cliques rooted at a vertex.} 
Let $G=K_k$, $k\ge2$, and $r=|R|=1$. Then $m_R(G)=k/2$  and $e_R(G)=k-1$.
To find $M$, consider first the candidates $H=K_2$ (with 
the root contained in $H$) and $H=G=K_k$.
For $H=K_2$, we have, as shown in general in comment
\ref{t3mn} above, 
$\lrpar{\Psi_{H}^R}^{1/\alpha^*(H-R)}=np$.
For $H=K_k$ we have
$\Psi_{K_k}=n^{k-1}p^{\binom k2}$ and $\ga^*(K_k-R)=\ga^*(K_{k-1})=(k-1)/2$,
and thus 
$\bigpar{\Psi_{K_k}^R}^{1/\alpha^*(K_k-R)}=n^2p^k$.
Hence,
\begin{equation}\label{ocab}
 M\le \min\bigcpar{ np,n^2p^k};
 \end{equation}
we will show that equality holds.

To this end, consider a general $H\subseteq G$ with $e(H-R)>0$
and let $F\=H-R$.
Then $e(H)\le e(F)+v(F)$ and so, see \eqref{Psi},
\begin{equation}\label{f1}
  \begin{split}
\frac{\Psi_H^R}{(np)^{\gax(H-R)}}
&\ge	\frac{n^{v(F)}p^{e(F)+v(F)}}{(np)^{\gax(F)}}
\\&
=\bigpar{np^{k-1}}^{v(F)-\gax(F)}p^{e(F)-(k-2)(v(F)-\gax(F))}
  \end{split}
\end{equation}
and, dividing \eqref{f1} by $\bigpar{np^{k-1}}^{\gax(H-R)}$,
\begin{equation}\label{f2}
  \begin{split}
\frac{\Psi_H^R}{(n^2p^k)^{\gax(H-R)}}
&\ge
\bigpar{np^{k-1}}^{v(F)-2\gax(F)}p^{e(F)-(k-2)(v(F)-\gax(F))}.
  \end{split}
\end{equation}
Since $\frac12v(F)\le\gax(F)\le v(F)$, we have
$v(F)-\gax(F)\ge0$ while $v(F)-2\gax(F)\le0$,
so 
$\bigpar{np^{k-1}}^{v(F)-\gax(F)}\ge1$ if $np^{k-1}\ge1$ and
$\bigpar{np^{k-1}}^{v(F)-2\gax(F)}\ge1$ if $np^{k-1}\le1$.
Further, by \cite[Lemma 6.1]{JOR}, since $F\subseteq G-R=K_{k-1}$, we have
$e(F)\le (k-2)(v(F)-\gax(F))$, and thus 
$p^{e(F)-(k-2)(v(F)-\gax(F))}\ge1$ for all
$p\in(0,1]$. Consequently, at least one of the right hand
sides of \eqref{f1} and \eqref{f2} is
$\ge1$, so 
\begin{equation*}
  \Psi_H^R\ge\min\lrcpar{(np)^{\gax(H-R)},(n^2p^k)^{\gax(H-R)}},
\end{equation*}
or
$(\Psi_H^R)^{1/\gax(H-R)}\ge\min\lrcpar{np,n^2p^k}$.
Finally, by \eqref{MRG} and \eqref{ocab},
\begin{equation*}
 M= \min\bigcpar{ np,n^2p^k}=
 \begin{cases}
   n^2p^k, & p\le n^{-1/(k-1)},
\\
   np, & p\ge n^{-1/(k-1)}.
 \end{cases}
 \end{equation*}
\end{example}

\begin{example} {\bf Bipartite graphs rooted at one whole side.}
 These are exactly the graphs with $e(G-R)=0$, and so $p_1=p_2$ 
(see comment \ref{t3c} after Theorem \ref{T3}). 
The two classes of the bipartition are $R$ and $S=V(G)\setminus R$. Since
the only connected subgraph 
of $G-R$ is $K_1$, and $\gax(K_1)=1$, we have from
\eqref{MRG} and the comments above that
 $M=np^{\Delta_S(G)},$ where $\Delta_S(G)\=\max_{v\in S}d_G(v)$ is the maximum
  degree in $G$ among all the vertices of $S$.
 Consequently, the upper bound in part (b) of Theorem \ref{T3} is of the form
$$\P(X^R_G\ge t\mu)\le\exp\{-\Theta(np^{\Delta_S(G)})\}.$$
\end{example}

It follows from the above example that the bounds on $\P(X^R_G\ge t\mu)$ for $K_{s,2}$ with $r=2$
and for even cycles $C_{2s}$ with $r=s$ are the same, since in both cases $\Delta_S(G)=2$ . This is
a special case of a more general phenomenon that the bounds depend only on the structure of $G-R$
and the degree sequence $|N_G(v)\cap R|$, $v\in V(G)\setminus R$. Our
next example provides one 
more instance of that.

\begin{example}%[Paths]
{\bf Paths rooted at the endpoints and cycles rooted at a vertex.} Let $G=P_k$ be a path with $k$
vertices, $k\ge3$,
and let $R$ be the set of its two endpoints. Then
$m_R(G)=\frac{k-1}{k-2}$, and so $p\ge 
n^{-1/m_R(G)}$ implies that $np\to\infty$ as $n\to \infty$. The
minimum in $M$ can be achieved only 
on a subpath $H$ on at most $k-2$  vertices containing one root, or $H=P_k$. So,
$$M=\min\left\{\min_{1\le l\le k-3}\left(n^lp^l\right)^{1/\lceil l/2\rceil},
\left(n^{k-2}p^{k-1}\right)^{1/\lceil (k-2)/2\rceil}\right\}.$$ The terms with even $l$ are all
equal to $(np)^2$ while for odd $l$  they are equal to $(np)^{2l/(l+1)}$, which means that the
smallest among them is $np$, the term corresponding to a single rooted
edge. Hence, for even $k$, 
$M=np$ if $p\ge n^{-(k-2)/k}$, and otherwise $M=n^2p^{2(k-1)/(k-2)}$, the term corresponding to
$H=G$. A similar cutoff for odd $k$ occurs at $n^{-(k-3)/(k-1)}$ with $M$ taking the values of
$n^{2(k-2)/(k-1)}p^2$ and $np$, in turn.

Finally, note that if $R'$ is a single vertex in a cycle $C_{k-1}$, $k\ge 4$, then
$m_{R'}(C_{k-1})=m_R(P_{k})$, $\Psi^{R'}_{C_{k-1}}=\Psi^R_{P_{k}}$,
$\gax(C_{k-1}-R')=\gax(P_{k}-R)$,
and the same is true for all other
candidates for the minimum in $M$, that is, paths with a root at one end. Thus,
$M_{R',C_{k-1}}=M_{R,P_k}$ and the upper tail bounds provided by Theorem \ref{T3} are the same for
these two rooted graphs.
\end{example}

\begin{remark}
  In the unrooted case,  the lower tails are typically much smaller than the upper tails (see
  Remark 8.3 in \cite{JOR}), and at best
  they can be of the same order of magnitude, e.g., when $p$ is fixed.
  Here, we encounter an opposite situation. Namely, for every $(R,G)$ with $e_R(G)>0$ and a
  fixed $p$, by the FKG inequality, we have for any $t>1$
  $$\P(X^R_G\le t\mu)\ge \P(X^R_G=0)\ge\P(e_R(\gnp)=0)=\exp\{-\Theta(n)\}$$
  while for $t>1$ and $p_1<p\le p_2$, by Corollary \ref{C3},
  $$\P(X_G^R\ge t\mu)=\exp\set{-\Theta(n^ 2)}.$$
\end{remark}

\begin{remark}
If there are no
 isolated vertices in $H-R$ and $n\ge v(H)$, $m\ge e(H)$, then
\eqref{NRN1} may be improved to
\begin{equation}\label{NRN2}
N^R(n,m,H)=\Theta(N(n,m,H-R))=\Theta(N(\min(n,2m),m,H-R)).
\end{equation}
Note, however, that this fails if $H$ contains a vertex whose all neighbors are among the roots;
for example if $H$ is a rooted edge and $n> m$, then $N^R(n,m,H)=m$ and $N(n,m,H-R)=n$.

For the lower bound in \eqref{NRN2},
take a graph $F_0$ (with $V(F_0)\cap R=\emptyset$)
which achieves the maximum in $N(n-r,m/{3r},H-R)$; we may assume that
$F_0$ has no isolated vertices, and thus at most $2m/(3r)$ vertices.
Then join all vertices of $R$ to all  vertices of $F_0$, obtaining a
graph $F_1$ which contains $R$, has at most $n$ vertices, at most
$m/(3r)+r2m/(3r)\le m$ edges, and is such that $N^R(F_1,H)\ge
N(F_0,H-R)$.
Hence, $N^R(n,m,H)\ge N(n-r,m/(3r),H-R)$.
Finally, provided $m\ge 3re(H-R)$, we use the fact proved in \cite{JOR} that
if $n'=\Theta(n)$, $m'=\Theta(m)$ and $n,n'\ge v(H)$,
$m,m'\ge e(H)$, then
$N(n',m',H)=\Theta(N(n,m,H))$
(this follows directly from  \cite[Theorem 1.3]{JOR}).
The case $e(H)\le m< 3re(H-R)$ is trivial, since then both sides of
\eqref{NRN2}  are $\Theta(1)$.
\end{remark}

\newcommand\AAP{\emph{Adv. Appl. Probab.} }
\newcommand\JAP{\emph{J. Appl. Probab.} }
\newcommand\JAMS{\emph{J. \AMS} }
\newcommand\MAMS{\emph{Memoirs \AMS} }
\newcommand\PAMS{\emph{Proc. \AMS} }
\newcommand\TAMS{\emph{Trans. \AMS} }
\newcommand\AnnMS{\emph{Ann. Math. Statist.} }
\newcommand\AnnPr{\emph{Ann. Probab.} }
\newcommand\CPC{\emph{Combin. Probab. Comput.} }
\newcommand\JMAA{\emph{J. Math. Anal. Appl.} }
\newcommand\RSA{\emph{Random Struct. Alg.} }
\newcommand\ZW{\emph{Z. Wahrsch. Verw. Gebiete} }
\newcommand\DMTCS{\jour{Discr. Math. Theor. Comput. Sci.} }

\newcommand\AMS{Amer. Math. Soc.}
\newcommand\Springer{Springer-Verlag}
\newcommand\Wiley{Wiley}

\renewcommand\vol{\textbf}
\newcommand\jour{\emph}
\newcommand\book{\emph}
\newcommand\inbook{\emph}
\def\no#1#2,{\unskip#2, no. #1,} %(typeset after year)
\newcommand\toappear{\unskip, to appear}

\newcommand\webcite[1]{\hfil
   \penalty0\texttt{\def~{{\tiny$\sim$}}#1}\hfill\hfill}
\newcommand\webcitesvante{\webcite{http://www.math.uu.se/~svante/papers/}}
\newcommand\arxiv[1]{\webcite{http://arxiv.org/#1}}

\def\nobibitem#1\par{}

\end{document}